\documentclass[12pt,reqno]{amsart}
\usepackage[title]{appendix}
\usepackage{bm}
\usepackage{cite}
\usepackage{calc}
\usepackage{cases}
\usepackage{comment}
\usepackage{enumerate}
\usepackage{geometry} 
\usepackage[linktocpage=true,colorlinks,citecolor=magenta,linkcolor=blue,urlcolor=magenta]{hyperref}
\usepackage{mathrsfs} 
\usepackage{url}
\usepackage{xcolor}

\theoremstyle{theorem}
\newtheorem{thm}{Theorem}[section]
\newtheorem{cor}[thm]{Corollary}

\newtheorem{lem}[thm]{Lemma}

\newtheorem{rem}[thm]{Remark}

\numberwithin{equation}{section}

\def\barlp{\overline{\Delta}}

\def\R{\mathbb{R}^{n+1}(c)}

\def\l{\langle}
\def\r{\rangle}

\begin{document}

\title[Triharmonic CMC hypersurfaces in space forms]{Triharmonic CMC hypersurfaces in space forms with at most 3 distinct principal curvatures}

\author[H. Chen ]{Hang Chen}
\address{School of Mathematics and Statistics, Northwestern Polytechnical University, Xi' an 710072,  P. R. China}
\email{\href{mailto:chenhang86@nwpu.edu.cn}{chenhang86@nwpu.edu.cn}}

\author[Z. Guan]{Zhida Guan}
\address{Department of Mathematical Sciences, Tsinghua University, Beijing 100084, P. R. China}
\email{\href{mailto:}{gzd15@mails.tsinghua.edu.cn}}

\begin{abstract}
A $k$-harmonic map is a critical point of the $k$-energy in the space of smooth maps between two Riemannian manifolds. In this paper, we prove that if $M^{n} (n\ge 3)$ is a CMC proper triharmonic hypersurface with at most three distinct principal curvatures in a space form $\mathbb{R}^{n+1}(c)$, then $M$ has constant scalar curvature. This supports the generalized Chen's conjecture when $c\le 0$.  When $c=1$, we give an optimal upper bound of the mean curvature $H$ for a non-totally umbilical proper CMC $k$-harmonic hypersurface with constant scalar curvature in a sphere. As an application, we give the complete classification of the 3-dimensional closed proper CMC triharmonic hypersurfaces in $\mathbb{S}^{4}$.
\end{abstract}



\maketitle

\section{Introduction}\label{sect:1}
A harmonic map $\phi$ between two Riemannian manifolds $(M,g)$ and $(N,\bar{g})$ is a critical point of the energy functional
\begin{equation*}
	E(\phi)=\frac{1}{2}\int_M |d\phi|^2 v_g,
\end{equation*}
where $\phi:(M,g)\rightarrow(N,\bar{g})$ is a smooth map. The Euler-Lagrange equations are given by the vanishing of the tension field $\tau(\phi)$. 

In 1964, Eells and Sampson \cite{ES1964} extended the concept of a harmonic map to a polyharmonic map of order $k$, that is, a critical point of the $k$-energy functional
\begin{equation*}
	E_k(\phi)=\frac{1}{2} \int_M \big|(d+d^{\ast})^k \phi\big|^2 v_g,
\end{equation*}
where $\phi:(M,g)\rightarrow(N,\bar{g})$ is a smooth map. The Euler-Lagrange equations are given by the vanishing of the $k$-tension field $\tau_k(\phi)$. 

For $k=2$, a critical point of the bi-energy $E_2$ is usually called a biharmonic map. In 1986, G. Jiang \cite{Jiang1986} studied the first and second variational formulae of $E_2$. We refer the readers to the very recent book by Ou and Chen \cite{OuChen} and a survey article by Fetcu and Oniciuc \cite{FO21} for the abundant progress  on biharmonic maps. For $k\ge 3$, the first and the second variational formulae of $E_k$ were obtained by S. Wang \cite{Wang1989} in 1989 and by Maeta \cite{Maeta2012a} in 2012 respectively. Polyharmonic maps have been rigorously studied as well, see \cite{ACL1983, BG1992,BL2002,Maeta2012,Maeta2012a,Maeta2015,NU2018,Ou2006} and the references therein. 

A harmonic map is always $k$-harmonic (cf. \cite[Proposition 3.1]{Wang1989}). When $\phi$ is the isometric immersion and $k$-harmonic, we usually call $M$ a $k$-harmonic submanifold of $N$. In this case, it is well known that $\tau(\phi)=n\mathbf{H}$, where $n$ is the dimension of $M$ and $\mathbf{H}$ is the mean curvature vector of $M$ in $N$. Hence, minimal submanifolds of $N$ are always $k$-harmonic ($k\ge 2$). Conversely, $k$-harmonic doesn't mean $l$-harmonic for $1\le l<k$ (cf. \cite{Wang1989, Maeta2012a}). For this reason, a $k$-harmonic map $\phi$ (resp. a submanifold $M$) is called ``\emph{proper}'' if $\phi$ (resp. $M$) is not harmonic (resp. minimal).

When $M$ is compact and $N$ is flat, by the expression of $\tau_k$ and the Stokes' theorem, it is not hard to show that $k$-harmonic and harmonic are equivalent (cf. \cite[Proposition 3.2]{Wang1989}). In the general case, B.-Y. Chen  \cite{Chen1991} proposed a famous conjecture in 1991: any biharmonic submanifold of the Euclidean space is minimal. Chen's conjecture is generalized by Maeta  \cite{Maeta2012}, which states that any polyharmonic submanifold of the Euclidean space is minimal. Both conjectures are still open, and several affirmative partial answers and related results have been obtained, see \cite{BFO2010,NU2011,NU2013,NUG2014,FHZ20, GLV21} and the references therein.

In this paper, we mainly focus on the case $k=3$. A critical point of $E_3$ is usually called a triharmonic map. Nowadays, the study of triharmonic maps is a particularly active subject (cf. \cite{Maeta2015a,MNU2015,MOR19,MR18,MR18a,Wang1991}). In 2015, Maeta-Nakauchi-Urakawa \cite{MNU2015} proved that under suitable conditions (e.g. the completeness of $M$, the finiteness of $E_4(\phi)$ and the $L^4$-norm of $\tau_\phi$), a triharmonic isometric immersion into a Riemannian manifold of non-positive curvature must be minimal. 
In 2012, Maeta proved that any compact constant mean curvature (CMC in short) triharmonic hypersurface $M^n$ in $\mathbb{R}^{n+1}(c) (c\leq 0)$ is minimal (see \cite[Proposition 4.3]{Maeta2012}); very recently,  Montaldo-Oniciuc-Ratto proved that the same conclusion holds for $n=2$ without the compactness assumption (see \cite[Theorem 1.3]{MOR19}). Here $\mathbb{R}^{n+1}(c)$ represents the $(n+1)$-dimensional Euclidean space $\mathbb{R}^{n+1}$, the hyperbolic space $\mathbb{H}^{n+1}$ and the unit sphere $\mathbb{S}^{n+1}$ for $c = 0, -1$ and $1$ respectively. Note that there are at most two distinct principal curvatures for a surface $M^{2}$. In this paper, we consider the higher dimension case and prove the following theorem. 

\begin{thm}\label{thm1.1}
	Let $M^{n} (n\ge 3)$ be a CMC proper triharmonic hypersurface with at most three distinct principal curvatures in $\R$. 
	Then the scalar curvature of $M$ is constant.
\end{thm}

	We would like to remind the readers that, the scalar curvature being constant is equivalent to the squared norm of the second fundamental form being constant for a CMC hypersurface in the space form, which can be derived from the Gauss equation \eqref{gauss}.
	Therefore, as a direct consequence of Theorem \ref{thm1.1}, we obtain
\begin{cor}\label{cor1}
	Let $M^{n} (n\ge 3)$ be a CMC triharmonic hypersurface with at most three distinct principal curvatures in $\R (c\le 0)$. Then $M$ must be minimal.
\end{cor}

When $c=1$, there are concrete examples of non-minimal $k$-harmonic hypersurfaces in a sphere, and some classifications and characterizations are obtained. The typical candidates are isoparametric ones. The small sphere $\mathbb{S}^{n}(a) (0<a<1)$ is totally umbilical and all of the principal curvatures are the same. The product of two spheres (sometimes called the \emph{generalized Clifford torus}) 
\begin{equation*}
\mathbb{S}^{m}(a)\times\mathbb{S}^{n-m}(\sqrt{1-a^{2}})\, (0<a<1)
\end{equation*}
is isoparametric of two distinct principal curvatures given, in some orientation, by
\begin{equation}\label{eq-prin-cur-1}
\lambda_1=\cdots=\lambda_{m}=\frac{\sqrt{1-a^{2}}}{a},\quad \lambda_{m+1}=\cdots=\lambda_{n}=-\frac{a}{\sqrt{1-a^{2}}}.
\end{equation}
However, in order to $\mathbb{S}^{n}(a)$ or $\mathbb{S}^{m}(a)\times\mathbb{S}^{n-m}(\sqrt{1-a^{2}})$ becomes $k$-harmonic, there are some restrictions on the parameters $a$ and $m$, see \cite{Maeta2015a, MR18} for details. 

For an orientable hypersurface $M$, we can choose an orientation for $M$ such that the mean curvature function, denoted by $H$, is non-negative. In the following statements, we always use this setting unless otherwise indicated.
Wang-Wu proved the following result.
\begin{thm}[{see \cite[Corollary 1.5]{WW12}}]\label{thm-WW}
	Let $M^{n} (n\ge 3)$ be a CMC proper biharmonic hypersurface in $\mathbb{S}^{n+1}$, then locally either
	
	(1) $H=1$ and $M=\mathbb{S}^{n}(1/\sqrt{2})$;
 or
 
 	(2) $H\in (0,\frac{n-2}{n}]$, and $H=\frac{n-2}{n}$ if and only if $M=\mathbb{S}^{n-1}(1/\sqrt{2})\times\mathbb{S}^{1}(1/\sqrt{2})$.
\end{thm}
Balmu\c{s}-Oniciuc proved a version of Theorem \ref{thm-WW} for proper biharmonic submanifolds with parallel mean curvature in $\mathbb{S}^{n+p}$ , which recovers Theorem \ref{thm-WW} when the codimension $p=1$, see \cite[Theorem 3.11]{BO12}.

  For $k\ge 3$, Montaldo-Oniciuc-Ratto proved that a CMC proper triharmonic surface in $\mathbb{S}^{3}$ is locally the totally umbilical sphere $\mathbb{S}^{2}(1/\sqrt{3})$ (see \cite[Theorem 1.6]{MOR19}).  In the same paper, they proved that 
\begin{thm}[{cf. \cite[Theorem 1.9]{MOR19}}]\label{thm-MOR}
	Let $M^{n} (n\ge 3)$ be a CMC proper $k$-harmonic ($k\ge 3$) hypersurface with constant scalar curvature in $\mathbb{S}^{n+1}$, then $H^{2}\in (0,k-1]$, and $H^{2}=k-1$ if and only if $M$ is locally the totally umbilical hypersurface $\mathbb{S}^{n}(1/\sqrt{k})$.
\end{thm}

Comparing Theorem \ref{thm-WW}, we can give an optimal upper bound for a hypersurface which is not totally umbilical and  satisfies the condition of Theorem \ref{thm-MOR}. Precisely, we prove that
\begin{thm}\label{thm-3}
	Let $M^{n} (n\ge 3)$ be a proper CMC  $k$-harmonic ($k\ge 3$) hypersurface  with constant scalar curvature in $\mathbb{S}^{n+1}$.  
	If $M$ is not totally umbilical, then $H^{2}\in(0,t_0]$, where
	$t_0$ is the largest real root of the polynomial 
	\begin{equation*}
	\begin{split}
		f_{n,k}(t)&=n^4 t^3-n^2\big((k-1)n^2-2(k+2) n+2(k+2)\big) t^2\\
		&\quad -(n-1)\big(3 n-(k+2)\big)\big((k-1) n-(k+2)\big) t-(n-1)(n-2)^2.
	\end{split}
	\end{equation*}
	
	Moreover,  $H^2=t_0$ if and only  $M$ is locally $\mathbb{S}^{n-1}(a)\times\mathbb{S}^{1}(\sqrt{1-a^{2}})$ with
	\begin{equation}\label{eq-rad}
	a^{2}=\frac{2(n-1)^{2}}{n^{2}H^{2}+2n(n-1)+nH\sqrt{n^{2}H^{2}+4(n-1)}}.
	\end{equation}
\end{thm}
\begin{rem}
	Since $f_{n,k}(t)$ is a cubic polynomial, either it has a unique real root and two non-real complex conjugate roots, or it has three  real roots (counted with multiplicity). 
	From the proof of Theorem \ref{thm-3} in Sect.~\ref{sect:4}, we can see that $t_0$ is simple and $t_0\in (0,k-1)$ in both cases.
	
	It is easy to see that $a^{2}<\frac{n-1}{n}$ from \eqref{eq-rad}, which is an interesting fact pointed out in \cite[Remark 1.7]{AC94}.
\end{rem}
Combining Theorem \ref{thm1.1} with Theorem \ref{thm-MOR} and Theorem \ref{thm-3}, we obtain
\begin{cor}\label{cor1.3}
	Let $M^{n} (n\ge 3)$ be a proper CMC  triharmonic hypersurface with at most three distinct principal curvatures in $\mathbb{S}^{n+1}$. 
	Then either 
	
	(1) $H^{2}=2$ and $M$ is locally $\mathbb{S}^{n}(1/\sqrt{3})$; or
	
	(2) $H^2\in (0, t_0]$, and $H^{2}=t_0$ if and only $M$ is locally $\mathbb{S}^{n-1}(a)\times\mathbb{S}^{1}(\sqrt{1-a^{2}})$. Here $t_0$ is the unique real root belonging to $(0, 2)$ of the polynomial 
	\begin{equation*}
	f_{n,3}(t)=n^4t^3-2n^2(n^2-5n+5)t^2-(n-1)(2n-5)(3n-5)t-(n-1)(n-2)^2,
	\end{equation*}
	and the radius $a$ is given by \eqref{eq-rad}.
	\end{cor}

We also have the following two consequences.
\begin{cor}\label{cor1.4}
	Let $M^{n} (n\ge 3)$ be a proper CMC  triharmonic hypersurface with two distinct principal curvatures in $\mathbb{S}^{n+1}$. 
	Then $M$ is locally  $\mathbb{S}^{m}(a)\times\mathbb{S}^{n-m}(\sqrt{1-a^{2}})$, where $a^{2}$ is the unique real root of the polynomial 
	\begin{equation}\label{eq-P}
	P_{n,m}(x):=3nx^3-(2n+5m)x^2+5mx-m.
	\end{equation}
\end{cor}
\begin{cor}\label{cor1.5}
	Let $M^{n} (n\ge 3)$ be a closed CMC triharmonic hypersurface with three distinct principal curvatures everywhere in $\mathbb{S}^{n+1}$. 
	Then $M$ must be minimal.
\end{cor}

In particular, we can obtain the complete classification of the 3-dimensional complete proper CMC triharmonic hypersurfaces in $\mathbb{S}^{4}$.
\begin{thm}\label{thm1.7}
	Let $M^{3}$ be a complete proper CMC triharmonic hypersurface in $\mathbb{S}^{4}$. 
	Then $M$ is one of the following: 
	
	(1) $M=\mathbb{S}^{3}(1/\sqrt{3})$; 
	
	(2) $M= \mathbb{S}^{2}(a)\times\mathbb{S}^{1}(\sqrt{1- a^{2}})$, where $a^{2}\approx 0.389833$ is the unique real root of $9x^{3}-16x^{2}+10x-2$.
\end{thm}

\begin{rem}
	A biharmonic version of Theorem \ref{thm1.7} was obtaind by Balmu\c{s}-Montaldo-Oniciuc (see \cite[Theorem 3.5]{BMO10}) under the ``compactness" assumption by using a classification result due to S. Chang \cite{Chang93a}, and then it was pointed out (cf. \cite[Theorem 3.10]{Oni12}, \cite[Theorem 4.6 and Remark 4.7]{FO21}) that the ``compactness" can be replaced with ``completeness" based on a classification theorem due to Cheng-Wan \cite{CW93}.
\end{rem}

The paper is organized as follows.
In Sect.~\ref{sect:2}, we introduce some notations and  recall some fundamental concepts and formulae for triharmonic hypersurfaces in space forms. In Sect.~\ref{sect:3}, we use proof by contradiction to prove Theorem \ref{thm1.1}. The main part is dealing with the case that $M$ has 3 distinct principal curvatures. In Sect.~\ref{sect:4}, we discuss the proper CMC triharmonic hypersurfaces in a sphere and prove the various results from Theorem \ref{thm-3} to Theorem \ref{thm1.7}.

\textbf{Acknowledgment}:
The first author was partially supported by Natural Science Foundation of Shannxi Province Grant No.~2020JQ-101 and the Fundamental Research Funds for the Central Universities Grant No.~310201911cx013.
The second author was partially supported by NSFC Grant No.~11831005 and No.~11671224.
The authors would like to thank Professor Haizhong Li for bringing the question to our attention and his valuable suggestions and comments. We are also grateful to Professor C. Oniciuc for giving helpful comments on the first version of the paper and bringing some references to our attention, which help us to add Corollary \ref{cor1.5} and improve Theorem \ref{thm1.7}.

\section{Preliminaries and notations}\label{sect:2}
\subsection{Fundamental formulae of hypersurfaces in 
$\R$}
Let $M$ be an $n$-dimensional hypersurface in the space form $\R$. We denote the Levi-Civita connections on $M$ and $\R$ by $\nabla$ and $\bar{\nabla}$ respectively.
For the tangent vector fields $X, Y, Z, W$ and the unit normal  vector field $\xi$ on $M$, the Riemannian curvature tensor of $M$, the Gauss and Weingarten formulae are respectively given by 
\begin{align*}
	R(X,Y,Z,W)&=\l R(X,Y)W, Z\r\\
	\bar{\nabla}_XY&=\nabla_XY+h(X,Y)\xi,\\
	\bar{\nabla}_{X}\xi&=-A(X).
\end{align*}
Here $	R(X,Y)Z=(\nabla_X\nabla_Y-\nabla_Y\nabla_X-\nabla_{[X,Y]})Z$, $h$ and $A$ represent the second fundamental form of $M$ and the shape operator respectively. 

In the rest of this paper, the ranges of indices  $i,j,k, \ldots$ is from 1 to $n$ except special declaration.

Now we choose an orthonormal frame $\{e_i\}_{i=1}^n$ of $M$ and suppose $\nabla_{e_i}e_j=\sum_k\Gamma_{ij}^{k}e_k$.

Denote $R_{ijkl}=R(e_i,e_j,e_k,e_l), h_{ij}=h(e_i,e_j), h_{ijk}=e_kh_{ij}-h(\nabla_{e_k}e_i,e_j)-h(e_i,\nabla_{e_k}e_j)$, then we have
\begin{align}
R_{ijkl}&=(e_i\Gamma_{jl}^{k}-e_j\Gamma_{il}^{k})+\sum_{m}(\Gamma_{jl}^{m}\Gamma_{im}^{k}-\Gamma_{il}^{m}\Gamma_{jm}^{k}-(\Gamma_{ij}^{m}-\Gamma_{ji}^{m})\Gamma_{ml}^{k}),\label{eq-curva}\\
h_{ijk}&=e_kh_{ij}-\sum_{l}(\Gamma_{ki}^{l}h_{lj}+\Gamma_{kj}^{l}h_{il}).\label{eq-hijk}
\end{align}
Now the Gauss-Codazzi equations can be written down as follows respectively:
\begin{align}\label{gauss}
R_{ijkl}&=(\delta_{ik}\delta_{jl}-\delta_{il}\delta_{jk})c+( h_{ik}h_{jl}-h_{il}h_{jk}),\\
\label{codazzi}
h_{ijk}&=h_{ikj}.
\end{align}

We denote $S=\sum\limits_{i,j}(h_{ij})^2$ the
 squared norm of the second fundamental form, $H=\frac{1}{n}
(\sum\limits_i h_{ii})$ the mean curvature  function of $M$.

\subsection{Triharmonic hypersurfaces in $\R$}

Consider $\phi: (M^{n},g)\to (N,\bar{g})$. 
From the first variational formula of $E_k$ (cf. \cite{Wang1989}), we have 
\begin{equation*}
\tau_3(\phi)=\barlp^{2}\tau(\phi)-\sum_{i}\bar{R}\big(\barlp\tau(\phi),d\phi(e_i)\big)d\phi(e_i)-\sum_{i}\bar{R}\big(\nabla^{\phi}_{e_i}\tau(\phi),\tau(\phi)\big)d\phi(e_i),
\end{equation*}
where $\nabla^{\phi}$ is the induced connection on the bundle $\phi^{-1}TN$ and  $\barlp=-\sum_{i}(\nabla^{\phi}_{e_i}\nabla^{\phi}_{e_i}-\nabla^{\phi}_{\nabla_{e_i}e_i})$ is the \emph{rough Laplacian} on the section of $\phi^{-1}TN$. Then by direct computation, one can derive that a CMC hypersurface of $\R$ is triharmonic if and only if (see \cite[Eq. (2.8)]{MOR19})
\begin{equation}\label{tri-cmc}
\begin{cases}
H(\Delta S+S^2-cnS-cn^{2}H^{2})=0, \\
HA(\nabla S)=0.
\end{cases}
\end{equation}
We remark that the sign of  $\Delta=-\sum_{i}(\nabla_{e_i}\nabla_{e_i}-\nabla_{\nabla_{e_i}e_i})$.

Obviously \eqref{tri-cmc} always holds for $H=0$, which is the trivial case. From now on, we assume that $M$ is not minimal, then \eqref{tri-cmc} becomes
\begin{subnumcases}{}
(\Delta S+S^2-cnS-cn^{2}H^{2})=0 \label{tri-1'}\\
A(\nabla S)=0.\label{tri-2'}
\end{subnumcases}

\section{Proof of Theorem \ref{thm1.1}}\label{sect:3}
In this section, we will prove Theorem \ref{thm1.1}. The strategy is that 
we suppose $\nabla S\neq 0$ and then derive a contradiction. 

If $\nabla S\neq 0$, then  \eqref{tri-2'}  implies that $\nabla S$ is a principal direction with the corresponding principal curvature $0$. Without loss of generality, we can choose an orthonormal frame $\{e_i\}$ such that $e_1$ is parallel to $\nabla S$ and the shape operator $A$ is diagonalized with respect to $\{e_i\}$, i.e.,  $h_{ij}=\lambda_i\delta_{ij}$, where $\lambda_i$ is the principal curvature and $\lambda_1=0$.

\subsection{At most two distinct principal curvatures}
If all principal curvatures of $M$ are the same, the $\lambda_i=0$ and $S=\sum_i\lambda_i^{2}=0$, which is a contradiction.

If $M$ has two distinct principal curvatures, then without loss of generality, we can assume that $\lambda_1=\cdots=\lambda_s=0, \lambda_{s+1}=\cdots=\lambda_n=\lambda\neq 0 (1\le s\le n-1)$. But $nH=(n-s)\lambda$ implies $\lambda$ is constant, so we obtain $S$ is constant, again a contradiction.

\subsection{Three distinct principal curvatures}
Next, we focus on the case that $M$ has three distinct principal curvatures. Without loss of generality, we suppose
\begin{equation*}
\lambda_1=\cdots=\lambda_r=0,\quad  \lambda_{r+1}=\cdots=\lambda_{r+s}=\lambda,\quad \lambda_{r+s+1}=\cdots=\lambda_{r+s+t}=\mu,
\end{equation*}
where $\lambda\neq \mu, \lambda\neq 0,\mu\neq 0; r\ge 1, s\ge 1, t\ge 1, r+s+t=n\ge 3$.

For convenience, we denote $I_1=\{1,\cdots,r\},I_2=\{r+1,\cdots,r+s\},I_3=\{r+s+1,\cdots,n\}$.
At first, we have
\begin{lem}\label{lem1}
Let $\alpha,\beta\in \{1,2,3\}$.
Then the coefficient $\Gamma_{ij}^{k}$ satisfies:
\begin{enumerate} 
	\item $\Gamma_{ij}^{k}=-\Gamma_{ik}^{j}$. \label{item:1}
	\item \label{item:2} $\Gamma_{ii}^{k}=\frac{e_k\lambda_i}{\lambda_i-\lambda_k}$ for $i\in I_\alpha$ and $k\notin I_\alpha$.
	\item $\Gamma_{ij}^{k}=\Gamma_{ji}^{k}$ if the indices satisfy one of the following conditions:
	\begin{enumerate}[(3a)] \label{item:3a}
		\item $i,j\in I_\alpha, k\notin I_\alpha$ ;
		\item $i,j\ge 2$ and $k=1$. \label{item:3b}
	\end{enumerate}
	\item $\Gamma_{ij}^{k}=0$ if the indices satisfy one of the following conditions:
	\begin{enumerate}[(4a)]
		\item $j=k$; \label{item:4a}
		\item $i=j\notin I_1, k\in I_1 $ and $k\ge 2$; \label{item:4b}
		\item $i=j\in I_1$ and $k\notin I_1$; \label{item:4c}
		\item $i,k\in I_\alpha, i\neq k$ and $j\notin I_\alpha$; \label{item:4d}
		\item $i,j\ge 2, i\in I_\alpha, j\in I_\beta $ with $\alpha\neq \beta$ and $k=1$. \label{item:4e}
	\end{enumerate}
\end{enumerate}

\end{lem}
\begin{proof}
	Item \eqref{item:1} is directly from $e_i\l e_j,e_k \r=0$, which also implies Item \eqref{item:4a}.
	
	By  \eqref{eq-hijk}, \eqref{codazzi} and the Item \eqref{item:1}, we have 
	\begin{equation}\label{eq-1}
	e_k\lambda_i=h_{iik}=h_{iki}=\Gamma_{ii}^{k}(\lambda_i-\lambda_k) \mbox{ for $i\neq k$,}
	\end{equation}
	\begin{equation}\label{eq-2}
	\Gamma_{ij}^{k}(\lambda_j-\lambda_k)=h_{kji}=h_{kij}=\Gamma_{ji}^{k}(\lambda_i-\lambda_k) \mbox{ for distinct $i,j,k$.}
	\end{equation}
	
	Item \eqref{item:2} follows from \eqref{eq-1} directly.
	
	Since $nH=s\lambda+t\mu$ is constant, we derive that for each $k$, 
	\begin{equation}\label{eq-3}
	se_k\lambda+te_k\mu=0.
	\end{equation} 
	On the other hand, $S=\sum_{i}\lambda_i^{2}=s\lambda^{2}+t\mu^2$, then for $k\ge 2$, $e_k S=0$ implies
	\begin{equation}\label{eq-4}
		s \lambda e_k\lambda + t \mu e_k\mu=0.
	\end{equation}
	From \eqref{eq-3} and \eqref{eq-4} it follows that 
	\begin{equation*}
	e_k\lambda=e_k\mu=0 \mbox{ for $k\ge 2$}
	\end{equation*} as $\lambda\neq \mu$. Hence, we obtain Item \eqref{item:4b} from Item \eqref{item:2}.
	
	Item \eqref{item:4c} is directly from \eqref{eq-1} since $e_k\lambda_i=0$ for $i\in I_1$.
	
	Item \eqref{item:3a} is directly from \eqref{eq-2} since $\lambda_j-\lambda_k=\lambda_i-\lambda_k\neq 0$.
 
	By the assumptions and the choice of $\{e_i\}$, we have $e_1S\neq 0, e_iS=0 (2\le i \le n)$. Then for $i,j\in\{2,\cdots,n\}, 0=(e_ie_j-e_ie_i)S=[e_i,e_j]S=(\nabla_{e_i}e_j-\nabla_{e_j}e_i)S=(\Gamma_{ij}^{1}-\Gamma_{ji}^{1})e_1S$, so $\Gamma_{ij}^{1}=\Gamma_{ji}^{1}.$ So we obtain Item \eqref{item:3b}.
	
	Again from \eqref{eq-2}, we have $\Gamma_{ij}^{k}(\lambda_j-\lambda_k)=0$ for $i,k\in I_\alpha, j\notin I_\alpha, i\neq k$. Since $\lambda_j-\lambda_k\neq 0$, we obtain Item \eqref{item:4d}. 
	
	If $i,j\ge 2$, then from \eqref{eq-2} and Item \eqref{item:3b} we can derive $\Gamma_{ij}^{1}(\lambda_j-\lambda_i)=0$. Since $i\in I_\alpha,j\in I_\beta$ and $\alpha\neq \beta$, we have $\lambda_j-\lambda_i\neq 0$ and then obtain Item \eqref{item:4e}. 
\end{proof}

\begin{lem}\label{lem2}
	Denote $P=\frac{e_1\lambda}{\lambda}, Q=\frac{e_1\mu}{\mu}$. Then we have
	\begin{equation}\label{eq-lem2}
	e_1P-P^{2}=c,\quad e_1Q-Q^{2}=c.
	\end{equation}
\end{lem}
\begin{proof}
	From \eqref{gauss} and \eqref{eq-curva}, for $j\in I_2$, we 
	have 
	\begin{equation}\label{eq-7}
	 c=R_{1j1j}=(e_1\Gamma_{jj}^{1}-e_j\Gamma_{1j}^{1})+\sum_{m}(\Gamma_{jj}^{m}\Gamma_{1m}^{1}-\Gamma_{1j}^{m}\Gamma_{jm}^{1}-\Gamma_{1j}^{m}\Gamma_{mj}^{1}+\Gamma_{j1}^{m}\Gamma_{mj}^{1}).
	\end{equation}
	By Items \eqref{item:1}, \eqref{item:2} and \eqref{item:4c} of Lemma \ref{lem1}, we have
	\begin{equation*}
	e_1\Gamma_{jj}^{1} -e_j\Gamma_{1j}^{1}=e_1P+e_j\Gamma_{11}^{j}=e_1P.
	\end{equation*} 
	By Items \eqref{item:4a}, \eqref{item:4b} and \eqref{item:4c} of Lemma \ref{lem1}, we have
	\begin{align*}
	\sum_m\Gamma_{jj}^{m}\Gamma_{1m}^{1}=\Gamma_{jj}^{1}\Gamma_{11}^{1}+\sum_{m\in I_1,m\ge 2}\Gamma_{jj}^{m}\Gamma_{1m}^{1}-\sum_{m\notin I_1}\Gamma_{jj}^{m}\Gamma_{11}^{m}=0.
	\end{align*}
	By Items \eqref{item:1}, \eqref{item:4a}, \eqref{item:4e} and \eqref{item:4d} of Lemma \ref{lem1}, we have
	\begin{align*}
	\sum_m\Gamma_{1j}^{m}\Gamma_{jm}^{1}=\sum_{m\ge 2, m\notin I_2}\Gamma_{1j}^{m}\Gamma_{jm}^{1}-\sum_{m\in I_2,m\neq j}\Gamma_{1j}^{m}\Gamma_{j1}^{m}=0.
	\end{align*}
	By Items \eqref{item:1}, \eqref{item:4a}, \eqref{item:4c}, \eqref{item:4e} and \eqref{item:4d} of Lemma \ref{lem1} of Lemma \ref{lem1}, we have
	\begin{align*}
	\sum_m\Gamma_{1j}^{m}\Gamma_{mj}^{1}=(\Gamma_{11}^{j})^{2}+\sum_{m\ge 2, m\notin I_2}\Gamma_{1j}^{m}\Gamma_{mj}^{1}-\sum_{m\in I_2,m\neq j}\Gamma_{1j}^{m}\Gamma_{m1}^{j}=0.
	\end{align*}
	By Items \eqref{item:1}, \eqref{item:4a}, \eqref{item:4e}, \eqref{item:4d} and \eqref{item:2} of Lemma \ref{lem1},
	\begin{align*}
	\sum_m\Gamma_{j1}^{m}\Gamma_{mj}^{1}=-\sum_{m\ge 2, m\notin I_2}(\Gamma_{jm}^{1})^{2}-\sum_{m\in I_2,m\neq j}(\Gamma_{j1}^{m})^{2}-(\Gamma_{jj}^{1})^{2}=-P^{2}.
	\end{align*}
	By putting the above equations into \eqref{eq-7}, we obtain the first equation of \eqref{eq-lem2}. The second equation of \eqref{eq-lem2} is from an analogous computation.	
\end{proof}

Now we are ready to derive the contradiction by  proving $\lambda$ is constant. 
Firstly, we have $e_1P=\frac{e_1e_1\lambda}{\lambda}-(\frac{e_1\lambda}{\lambda})^{2}=\frac{e_1e_1\lambda}{\lambda}-P^{2}$, then the first equation of \eqref{eq-lem2} gives
\begin{equation}\label{eq3-1}
	e_1e_1\lambda-2\lambda P^{2}=c\lambda.
\end{equation}
Similarly, the second equation of \eqref{eq-lem2} gives
\begin{equation}\label{eq3-2}
e_1e_1\mu-2\mu Q^{2}=c\mu.
\end{equation}
	Since $nH=s\lambda+t\mu$ is constant, 
we have
\begin{equation}\label{eq3-3}
se_1\lambda+te_1\mu=s\lambda P+t\mu Q=0.
\end{equation}
From \eqref{eq3-1} and \eqref{eq3-2}, we derive
\begin{equation}\label{eq3-4}
s\lambda P^{2}+t \mu Q^{2}=-cnH/2.
\end{equation}
Eliminating $Q$ in \eqref{eq3-4} by using \eqref{eq3-3}, we obtain
\begin{equation*}
-cnH/2=s\lambda P^{2}+t\mu (\frac{s\lambda}{t\mu})^{2}P^{2}=s\lambda P^{2}(1+\frac{s\lambda}{t\mu})=s\lambda P^{2}(\frac{nH}{t\mu}),
\end{equation*}
which is equivalent to
\begin{equation}\label{eq3-5}
	 \frac{s\lambda}{t\mu}P^{2}=-c/2
\end{equation}
as $H\neq 0$.
Differentiating \eqref{eq3-5} along $e_1$, by using $e_1\lambda=\lambda P, e_1P=c+P^{2}$ and \eqref{eq3-5}, we derive
\begin{align}
	0&=e_1\Big(\frac{s\lambda}{nH-s\lambda}\Big)P^{2}+2\Big(\frac{s\lambda}{nH-s\lambda}\Big)Pe_1P\nonumber\\
	&=\frac{s(nH-s\lambda)e_1\lambda+s\lambda se_1\lambda}{(nH-s\lambda)^{2}}P^{2}+2\Big(\frac{s\lambda}{nH-s\lambda}\Big)P(c+P^{2})\nonumber\\
	&=\frac{se_1\lambda }{(nH-s\lambda)^{2}}\Big[(nH-s\lambda) P^{2}+s\lambda P^{2}+2(nH-s\lambda)(c+P^{2})\Big]\nonumber\\
	&=\frac{se_1\lambda }{(nH-s\lambda)^{2}}(t\mu(2c+3P^{2})-t\mu c/2).\label{eq3-8}
\end{align}
If $e_1\lambda=\lambda P= 0$, then both $\lambda$ and $\mu$ are constant, a contradiction to the initial assumption $\nabla S\neq 0$. Hence, \eqref{eq3-8} is equivalent to 
\begin{equation*}
3P^{2}+3c/2=0,
\end{equation*}
which means $P^{2}$ equals to a constant $-c/2$.
Now the first equation of \eqref{eq-lem2} becomes $0-P^{2}=c$ as $P$ is constant, which is impossible when $c\neq 0$. But when $c=0$, we have $P=0$. We obtain the contradiction again.

In conclusion, we complete the proof of Theorem \ref{thm1.1}.

Corollary \ref{cor1} follows from \eqref{tri-1'} by applying Theorem \ref{thm1.1}.
 
\section{Proper $k$-harmonic hypersurfaces in a sphere}\label{sect:4}
In this section, we study proper CMC $k$-harmonic  hypersurfaces in a sphere. 
Keep in mind that ``proper'' means $H\neq 0$. 
Theorem 1.10 of \cite{MOR19} says that a proper CMC hypersurface with constant scalar curvature in $\mathbb{S}^{n+1}$ is $k$-harmonic $(k\ge 3)$ if and only if 
\begin{equation*}
	S^{2}-nS-(k-2)n^{2}H^{2}=0,
\end{equation*}
from which we can solve out $S$ and obtain
\begin{equation}\label{eq4-1}
S=\frac{n}{2}\big(1+\sqrt{1+4(k-2)H^{2}}\big).
\end{equation}
By Cauchy-Schwarz inequality, $S\ge nH^{2}$. Putting this into \eqref{eq4-1}, we derive that
\begin{equation}\label{eq4-2}
\sqrt{1+4(k-2)H^{2}}\ge 2H^{2}-1.
\end{equation}
When $H^{2}\le 1/2$, \eqref{eq4-2} obviously holds. When $H^{2}> 1/2$, \eqref{eq4-2} is equivalent to
\begin{equation*}
1+4(k-2)H^{2}\ge (2H^{2}-1)^{2}=4H^{4}-4H^{2}+1,
\end{equation*}	
which implies $H^{2}\le k-1$. Moreover, it is easy to see that $S=nH^{2}$ if and only if $H^{2}=k-1$.
So far, we have given a quick review of the proof of Theorem \ref{thm-MOR}, and we see that if $M$ is not totally umbilical (i.e., $S>nH^{2}$), then $H^{2}<k-1$. 

To prove Theorem \ref{thm-3}, we need the following lemma.
	\begin{lem}\label{lem-rigidity}
			Let $M$ be an $n$-dimensional non-minimal hypersurface in $\mathbb{S}^{n+1}$ with both $H$ and $S$ are constant. If
		\begin{equation}\label{cond-2}
		nH^{2}< S\le n+\frac{n^{3}H^{2}}{2(n-1)}-\frac{n(n-2)}{2(n-1)}\sqrt{n^{2}H^{4}+4(n-1)H^{2}},
		\end{equation}
		then 
		\begin{equation}\label{result-1}
		S= n+\frac{n^{3}H^{2}}{2(n-1)}-\frac{n(n-2)}{2(n-1)}\sqrt{n^{2}H^{4}+4(n-1)H^{2}},
		\end{equation}
			and locally $M=\mathbb{S}^{n-1}(a)\times\mathbb{S}^{1}(\sqrt{1-a^{2}})$ with the principal curvatures given by
		\begin{equation}\label{eq-pri-cur}
		\lambda_1=\cdots=\lambda_{n-1}=\frac{nH+\sqrt{n^{2}H^{2}+4(n-1)}}{2(n-1)}, \quad \lambda_n=\frac{nH-\sqrt{n^{2}H^{2}+4(n-1)}}{2}.
		\end{equation} 
	\end{lem}
	\begin{proof}
		Essentially this lemma is a special case of a rigidity result due to H. Li \cite[Corollary 2.1]{Li1997}, and we reproduce the main steps in H. Li's proof  for the readers' convenience. We also remark that, both $H$ and $S$ being constant is a strong condition so that we don't need the assumption of compactness of $M$ which is required in H. Li's original results.
		
		We choose an orientation such that $H\ge 0$. Now consider the operator $\square$ defined by 
		\begin{equation*}
			\square f=\sum_{i,j}\Big(nH\delta_{ij}-h_{ij}\Big)f_{ij} \text{ for $f\in C^{2}(M)$},
		\end{equation*}
		then we have (cf. \cite[Eq. (1.17), Eq. (2.14) and Eq. (2.15)]{Li1997})
		\begin{align*}
		\square (nH)&=-\frac{1}{2}\Delta (n^{2}H^{2}-S)+|\nabla A|^{2}-|\nabla(nH)|^{2}+\frac{1}{2}\sum_{i,j}R_{ijij}(\lambda_i-\lambda_j)^{2}\\
		&\ge -\frac{1}{2}\Delta (n^{2}H^{2}-S)+|\nabla A|^{2}-|\nabla(nH)|^{2}+(S-nH^{2})\\
		&\qquad \times\left(\sqrt{S-nH^{2}}+\frac{n-2}{2}\sqrt{\frac{n}{n-1}}H+\sqrt{n+\frac{n^{3}H^{2}}{4(n-1)}}\right)\nonumber\\
		&\qquad\times\left(-\sqrt{S-nH^{2}}-\frac{n-2}{2}\sqrt{\frac{n}{n-1}}H+\sqrt{n+\frac{n^{3}H^{2}}{4(n-1)}}\right)
		\end{align*}
		Since $H$ and $S$ are constant, considering \eqref{cond-2}, the above inequality becomes 
		\begin{align*}
		0&\ge |\nabla A|^{2}+(S-nH^{2})\\
		&\qquad \times\left(\sqrt{S-nH^{2}}+\frac{n-2}{2}\sqrt{\frac{n}{n-1}}H+\sqrt{n+\frac{n^{3}H^{2}}{4(n-1)}}\right)\\
		&\qquad\times\left(-\sqrt{S-nH^{2}}-\frac{n-2}{2}\sqrt{\frac{n}{n-1}}H+\sqrt{n+\frac{n^{3}H^{2}}{4(n-1)}}\right)\ge 0.
		\end{align*}
		Hence, \eqref{result-1} holds and all the inequalities above must become equations. Then we have $|\nabla A|^{2}=0$ and there are $(n-1)$ of the principal curvatures $\lambda_{i}$'s are the same (cf. \cite[Lemma 1.3]{Li1997}). Without loss of generality, we suppose $\lambda_{1}=\cdots=\cdots\lambda_{n-1}\neq\lambda_{n}$. Then \eqref{eq-pri-cur} follows from $(n-1)\lambda_{1}+\lambda_{n}=nH$ and $R_{1n1n}=\lambda_{1}\lambda_{n}+1=0$, and the readers can find the details in 
		\cite[Sect.~2]{Li1997}. The radius $a$ can be determined by \eqref{eq-prin-cur-1} and \eqref{eq-pri-cur}.
	\end{proof}

	\begin{proof}[Proof of Theorem \ref{thm-3}]  
	According to Lemma \ref{lem-rigidity}, it is sufficient to find suitable values of $H$ such that 
	\begin{equation}\label{cond-2'}
	S\ge n+\frac{n^{3}H^{2}}{2(n-1)}-\frac{n(n-2)}{2(n-1)}\sqrt{n^{2}H^{4}+4(n-1)H^{2}}.
	\end{equation} 
	For convenience, we denote $H^{2}=t$.
	Recall \eqref{eq4-1}, we obtain \eqref{cond-2'} holds if and only 
	\begin{equation}\label{eq4-4}
		(n-1)\sqrt{1+4(k-2)t}+(n-2)\sqrt{n^{2}t^{2}+4(n-1)t}\ge (n-1)+n^{2}t.
	\end{equation}
	Square both sides of \eqref{eq4-4} and simplify it, then we obtain
	\begin{equation}\label{eq4-5}
		(n-2)\sqrt{1+4(k-2)t}\sqrt{n^{2}t^{2}+4(n-1)t}\ge \Big(2n^{2}t-\big(n^{2}+2(n-1)(k-6)\big)\Big)t
	\end{equation}
	Define 
	\begin{equation*}
	\begin{split}
		f_{n,k}(t)&:=n^4 t^3-n^2\big((k-1)n^2-2(k+2) n+2(k+2)\big) t^2\\
		&\quad -(n-1)\big(3 n-(k+2)\big)\big((k-1) n-(k+2)\big) t-(n-1)(n-2)^2,
	\end{split}
	\end{equation*}			
	and set
	\begin{equation*}
	t_1:=\frac{n^{2}+2(n-1)(k-6)}{2n^{2}}.
	\end{equation*}	
	Clearly, $t_1<\frac{1}{2}+\frac{k}{n}\le k-1$ as $k\ge 3$. 
	Note that 
	\begin{equation*}
	f'_{n,k}(t)=3n^{4}t^{2}-2n^{2}((k-1)n^{2}-2(k+2)n+2(k+2))t-(n-1)(3n-(k+2))((k-1)n-(k+2)),
	\end{equation*}
	then we obtain
	\begin{equation*}
	f_{n,k}(t_1)=-\frac{(n-2)^{2}}{8 n^{2}}\left((n-2)^{2}+2(n-1) k\right)\left(4(n-1)(k-6)(k-2)+n^{2}(2 k-3)\right)<0,
	\end{equation*}	
	\begin{equation*}
	f'_{n,k}(t_1)=-\frac{1}{4}(n-2)^{2}\left(8(n-1)(k-4)(k-2)+n^{2}(4 k-7)\right)<0,
	\end{equation*}	
	\begin{align*}
	f_{n,k}(k-1)&=k^2(n-1)((n-1)(k+3)+n^{2}(2 k-3))>0,\\
	f'_{n,k}(k-1)&=(n-2)^{2}(n^{2}-n+1)+n^{4}k(k-2)+3n^{2}(n-1)k^{2}\\
	&\quad +n(n^{2}-2)k^{2}+nk(n^{2}+3n-8)+k(k+4)>0,\\
	f_{n,k}(0)&=-(n-1)(n-2)^2<0.
	\end{align*}	

	Since $f_{n,k}(t)$ is a cubic function, there is a unique real root $t_0\in (t_1,k-1)$, and $t_0$ is the largest one of all possible real roots of $f_{n,k}(t)$. Moreover, since $f_{n,k}(0)<0$,  we always have  $t_0\in(0,k-1)$ whenever $t_1<0$ or $t_1\ge 0$.
	
	If $0< t< t_1$, then the inequality \eqref{eq4-5} strictly holds since the left side of \eqref{eq4-5} is positive while the right side is negative. 
		
	If $t\ge \max\{0,t_1\}$,  
	by squaring both sides of \eqref{eq4-5} and simplifying it, then we obtain
	\begin{equation*}
		(n-2)^{2}(1+4(k-2)t)\big(n^{2}t^{2}+4(n-1)t\big)\ge \Big(2n^{2}t-\big(n^{2}+2(n-1)(k-6)\big)\Big)^{2}t^{2},
	\end{equation*}
	which is equivalent to 	$f_{n,k}(t)\le 0.$	Hence, \eqref{eq4-5} holds in this case if and only if $\max\{0,t_1\}\le t\le t_0$, and ``$=$'' holds in \eqref{eq4-5} if and only if $t=t_0$.
	
	In summary, we obtain \eqref{cond-2'} holds if and only if $H^{2}\in (0,t_0]$, and ``$=$'' holds in \eqref{cond-2'} if and only if $H^{2}=t_0$. Combining with Lemma \ref{lem-rigidity}, we complete the proof.
\end{proof}

Corollary \ref{cor1.3} follows directly from Theorem \ref{thm1.1}, Theorem \ref{thm-MOR} and Theorem \ref{thm-3}, and we omit its proof. 

\begin{proof}[Proof of Corollary \ref{cor1.4}]
	We denote the two distinct principal curvatures, then we can assume $\lambda_1=\cdots=\lambda_{m}=\lambda, \lambda_{m+1}=\cdots=\lambda_{n}=\mu$. By Theorem \ref{thm1.1}, we derive that $S=m\lambda^{2}+(n-m)\mu^{2}$ is constant since $nH=m\lambda+(n-m)\mu$ is constant. Now $\lambda$ satisfies the quadratic equation of constant coefficients $mn\lambda^{2}-2mnH\lambda+n^{2}H^{2}-(n-m)S=0$, which implies that both $\lambda$ and $\mu$ are constant. So locally $M=\mathbb{S}^{m}(a)\times\mathbb{S}^{n-m}(\sqrt{1-a^{2}})$ by the classification theorem of isoparametric hypersurfaces due to E. Cartan \cite{Cartan38}. From \eqref{eq-prin-cur-1}, we have 
	\begin{equation*}
		nH=\frac{m-na^{2}}{a\sqrt{1-a^{2}}},\quad S=\frac{na^{4}-2ma^{2}+m}{a^{2}(1-a^{2})}.
	\end{equation*}
	Hence, \eqref{tri-1'} becomes
	\begin{equation*}
	\frac{(na^{2}-m)\big(3na^{6}-(2n+5m)a^{4}+5ma^{2}-m\big)}{\big(a^{2}(1-a^{2})\big)^{2}}=0.
	\end{equation*}
	We get rid of $a^{2}=m/n$ because it corresponds with the minimal case $H=0$. Hence, $a^{2}$ is the root of $P_{n,m}(x)$ defined by \eqref{eq-P}.
	
	At last, we point out $P_{n,m}(x)$ has only one real root since its discriminant is $-m(n-m)(32n^{2}-125mn+125m^{2})<0$, and this unique real root lies in $(0,1)$ since  $P_{n,m}(0)=-m<0,P_{n,m}(1)=n-m>0$. Hence, we complete the proof.
\end{proof}
\begin{rem}
	 The sufficient and necessary condition for a generalized Clifford torus being  $k$-harmonic was obtained by Maeta \cite{Maeta2015a} for $k=3$ and by Montaldo-Ratto \cite{MR18} for $k\ge 4$.
\end{rem}

\begin{proof}[Proof of Corollary \ref{cor1.5}]
	If $H\neq 0$, then from Theorem \ref{thm1.1} we conclude $M$ has constant scalar curvature, which follows that  $M$ must be isoparametric according to a result of S. Chang \cite{Chang94}. But there is no any proper triharmonic isoparametric hypersurface in $\mathbb{S}^{n+1}$ (cf. \cite[Theorem 1.12]{MOR19} or the following proof of Theorem \ref{thm1.7} for details), which is a contradiction. Hence, $M$ must be minimal.
\end{proof}

\begin{rem}
	From Corollary \ref{cor1.3} we can only conclude that $|H|^2 < t_0$ in hypotheses of Corollary \ref{cor1.5}, so we cannot get the contradiction in this way.
\end{rem}

\begin{proof}[Proof of Theorem \ref{thm1.7}]
	We recall a classification theorem due to Cheng-Wan \cite{CW93}, which states that, any complete CMC hypersurface with constant scalar curvature in $\mathbb{S}^{4}$ is isoparametric. It is well known that the classification of isoparametric hypersurfaces with at most 3 distinct principal curvatures in a sphere were obtained by E. Cartan \cite{Cartan38}, so we can verify if each case satisfies \eqref{tri-1'}.
	
	If all the principal curvatures of $M$ are the same, then $M$ is a hypersphere $\mathbb{S}^{3}(a)$, which is totally geodesic. A simple calculation from \eqref{tri-1'} shows that $a$ must be $1/\sqrt{3}$. This is Item (1).
	
	If $M$ has two distinct principal curvatures, then we have Item (2) by taking $n=3, m=2$ in  Corollary \ref{cor1.4}.
	
	If $M$ has three distinct principal curvatures, then the principal curvatures are given by (cf. \cite{Cartan38}):
	\begin{equation*}
	\lambda_{1}=\cot s,\lambda_{2}=\cot\Big(s-\frac{\pi}{3}\Big),\lambda_{3}=\cot\Big(s+\frac{\pi}{3}\Big), \quad s\in \Big(0,\frac{\pi}{3}\Big).
	\end{equation*}
	Hence, after a long but direct calculation,  \eqref{tri-1'} becomes
	\begin{equation*}
	\frac{9\csc^{4}(s)(3\cos(12s)+28\cos(6s)+41\big)}{8\big(1+2\cos(2s)\big)^{2}}=0.
	\end{equation*}
	It follows that $\cos(6s)$ should satisfies $3x^{2}+14x+19=0$, but this quadratic equation has no real roots. So we obtain that $M$ is not triharmonic in this case.
\end{proof}
\begin{rem}
	The exclusion of the last case in the above proof can be directly obtained from \cite[Theorem 1.12]{MOR19}. Actually,  
	Montaldo-Oniciuc-Ratto \cite{MOR19} discussed  all of the possibilities that an isoparametric hypersurface in $\mathbb{S}^{n+1}$ becomes  proper $k$-harmonic for suitable values of $k$. 
\end{rem}

\end{document}